\documentclass[11pt,a4paper]{article}
\usepackage{amssymb,amsmath,latexsym,euscript}
\usepackage{a4wide}
\usepackage{parskip}
\usepackage{float}
\usepackage{longtable}
\usepackage{booktabs}
\usepackage{listings}
\usepackage{placeins}
\usepackage{enumitem}
\usepackage{xcolor}

\usepackage{hyperref}
\usepackage{booktabs}
\usepackage{caption}
\usepackage{siunitx}
\usepackage{tabularx}
\usepackage{graphicx}
\usepackage{adjustbox}
\usepackage{multirow}
\usepackage{amsthm}
\usepackage{bm}
\usepackage{lscape}
\usepackage{float}
\usepackage[noadjust]{cite}

\def\qed{\hfill {$\square$}\goodbreak \medskip}

\newtheorem{theorem}{Theorem}[section]
\newtheorem{lemma}[theorem]{Lemma}
\newtheorem{corollary}[theorem]{Corollary}
\theoremstyle{definition}

\newtheorem{example}[theorem]{Example}
\newtheorem{proposition}[theorem]{Proposition}
\newtheorem{remark}[theorem]{Remark}

\numberwithin{equation}{section}

\usepackage{tikz,xcolor,hyperref}

\def\x{{\mathbf x}}

\def\u{{\mathbf u}}

\def\v{{\mathbf v}}
\def\c{{\mathbf c}}
\newcommand{\F}{\mathbb{F}}

\def\C{{\mathcal{C}}}

\def\H{{\mathit{H}}}

\title{
Symplectic  Quasi Self-dual, Self-dual, and LCD Codes over Non-unital Rings of Order Six}
\author{Altaf Alshuhail\thanks{Department of Mathematics, Faculty of Science, University of Ha’il, Ha’il, Saudi Arabia, email: \tt ad.alshuhail@uoh.edu.sa} }
\date{}

\begin{document}

\maketitle

\begin{abstract}
We consider codes over the two semi-local non-unital rings of order six,
{\small \[ H_{23} = \langle a,b \mid 2a=0, 3b = 0, a^2=a, b^2 = 0, ab = 0 = ba \rangle,\]}
and
{\small \[H_{32} = \langle a,b \mid 2a=0, 3b = 0, a^2=0, b^2 = b, ab = 0 = ba \rangle. \]}
with respect to a symplectic inner product. Via the  decomposition \(\C=a\,\C_a\oplus b\,\C_b\), any \(H_z\)-code splits into a binary component \(\C_a\) and a ternary component \(\C_b\); this yields  characterizations of symplectic self-orthogonal, self-dual, and quasi self-dual (QSD) codes, and allows the introduction of symplectic nice and symplectic linear complementary dual (LCD) codes. Using the automorphism groups of \(\C_a\) and \(\C_b\) and double-coset enumeration, we classify, up to permutation equivalence, all symplectic self-orthogonal and QSD \(H_z\)-codes of short even lengths. 
\end{abstract}

\section{Introduction}
Codes over rings were first investigated in the early 1990's as a generalization of classical linear codes over fields~\cite{ShiAlahmadiSole2017}. Since 2002, codes over \(\mathbb{Z}_6\) have attracted attention for their connections to Euclidean lattices, ring-linear constructions, and duality frameworks~\cite{Z61,Z62}. Self-dual codes have attracted considerable attention in coding theory and have been extensively studied over both finite fields and rings \cite{HuffmanKimSole2021,MacWilliamsSloane1977,NebeRainsSloane2006,PlessBrualdiHuffman1998}. A suitable analogue of self-duality for non-unital rings of order four is given by quasi-self-dual (QSD) codes with respect to the Euclidean inner product: these are self-orthogonal codes of length \(n\) and size \(2^n\). The notion, developed in \cite{Alahmadi2020,Alahmadi2021,Alahmadi2022}, is motivated by the failure of the classical size relation \(|\C|\,|\C^\perp|=|A|^n\) (and thus \(|\C|=|\C^\perp|\) for self-dual \(\C\)) over non-unital rings.  Self-orthogonal, self-dual, and QSD codes over non-unital rings of orders four, six, and \(p^2\) (for odd primes \(p\)) have been constructed and classified under Euclidean duality~\cite{Altaf,AAABBSS,AE,A1,A2,A3}. More recently,  symplectic inner products have been introduced to define symplectic self-orthogonality and symplectic self-duality over both unital and non-unital rings~\cite{sara1,sara2}. Furthermore, linear complementary dual (LCD) codes—linear codes with complementary duals—were introduced by Massey for codes over finite fields ~\cite{SLCD} and have since attracted significant interest because of their cryptographic applications ~\cite{CarletGuilley2016}. Over non-unital rings, LCD codes can be defined with respect to two inner products (see,  \cite{SLCD,Melaibari.A et al.} for the Euclidean setting) and  (see,  \cite{sara1} for the symplectic setting).

In this paper, we study symplectic codes over the two semi-local non-unital rings of order six, \(H_{23}\) and \(H_{32}\). The paper is structured as follows. Section~\ref{sec:1} collects preliminary definitions and facts about symplectic self-orthogonal, symplectic self-dual, quasi self-dual, symplectic nice, and symplectic LCD codes over the fields \(\F_2\) and \(\F_3\). Section~\ref{sec:rings} is devoted to the study of codes over the rings \(H_{23}\) and \(H_{32}\), including their algebraic properties, code constructions, and duality relations. Section~\ref{sec:duality} develops the theory of symplectic duality for \(H_z\)-codes, providing characterizations of self-orthogonal, self-dual, quasi self-dual, symplectic nice, and LCD codes in terms of their binary and ternary components, as well as explicit construction methods.  Section~\ref{Hclass} classifies permutation-inequivalent symplectic self-orthogonal \(H_z\)-codes of even length \(n\le 8\).

\section{Binary and Ternary Codes}\label{sec:1} 

Let \(m \ge 0\) be an integer, and write \(\mathbb{F}_2^{2m}\) for the \(2m\)-dimensional vector space over the binary field \(\mathbb{F}_2\). Write each vector in the form $\mathbf{v} = (\mathbf{v}_1 \mid \mathbf{v}_2),\, \mathbf{v}_1,\mathbf{v}_2\in\F_2^m.$
Define the \emph{symplectic inner product} on \(\F_2^{2m}\) by
\begin{equation}
\langle \mathbf{x}, \mathbf{y} \rangle_s^{(\F_2)}
\;=\;
\mathbf{x}_1 \cdot \mathbf{y}_2
\;+\;
\mathbf{x}_2 \cdot \mathbf{y}_1.
\end{equation}

where “\(\cdot\)” is the usual dot product in \(\F_2^m\).  A \emph{binary linear code} \(\C\) of length \(n=2m\) and dimension \(k\) is an \(\F_2\)-subspace of \(\F_2^n\) with \(\dim_{\F_2}\C=k\).  Equivalently, one calls \(\C\) an \([n,k]\)-code.  Its \emph{symplectic dual} is
\[
\C^{\perp_s}
\;=\;
\bigl\{
\mathbf{y}\in\F_2^{2m}
\;\bigm|\;
\langle \mathbf{x},\mathbf{y}\rangle_s^{(\F_2)}=0
\;\forall\,\mathbf{x}\in\C
\bigr\}.
\]

Similarly, let \(\F_3^{2m}\) be the \(2m\)-dimensional vector space over the ternary field \(\F_3\), and write $\mathbf{v} = (\mathbf{v}_1 \mid \mathbf{v}_2),\,\mathbf{v}_1,\mathbf{v}_2\in\F_3^m.$
Define the \emph{symplectic inner product} on \(\F_3^{2m}\) by
\begin{equation}
\langle \mathbf{x}, \mathbf{y} \rangle_s^{(\F_3)}
\;=\;
\mathbf{x}_1 \cdot \mathbf{y}_2
\;-\;
\mathbf{x}_2 \cdot \mathbf{y}_1.
\end{equation}

where “\(\cdot\)” is the dot product in \(\F_3^m\).  A \emph{ternary linear code} \(\C\) of length \(2m\) and dimension \(k\) is an \(\F_3\)-subspace of \(\F_3^n\) with \(\dim_{\F_3}\C=k\).  Equivalently, \(\C\) is an \([2m,k]\)-code, and its symplectic dual is
\[
\C^{\perp_s}
\;=\;
\bigl\{
\mathbf{y}\in\F_3^{2m}
\;\bigm|\;
\langle \mathbf{x},\mathbf{y}\rangle_s^{(\F_3)}=0
\;\forall\,\mathbf{x}\in\C
\bigr\}.
\]

For \(p\in\{2,3\}\) and with respect to the symplectic inner product, we call a code \(\C\) \emph{symplectic self-orthogonal}(\emph{symplectic-SO} ) if \(\C\subseteq \C^{\perp_s}\), \emph{symplectic self-dual} (\emph{symplectic-SD} ) if \(\C=\C^{\perp_s}\), and \emph{symplectic-LCD} (linear code with complementary dual) if \(\C\cap \C^{\perp_s}=\{0\}\). Two codes \(\C\) and \(\C'\) are \emph{permutation equivalent} if there exists a permutation of the \(n\) coordinate positions sending \(\C\) to \(\C'\). The elements of \(\C\) are called \emph{codewords}.


\section{Codes over the rings $H_{23}$ and $H_{32}$}\label{sec:rings}

In this paper, we consider two rings which are semi-local non-unital rings of order six. We denote these rings by {\small \[ H_{23} = \langle a,b \mid 2a=0, 3b = 0, a^2=a, b^2 = 0, ab = 0 = ba \rangle,\]}
and
{\small \[H_{32} = \langle a,b \mid 2a=0, 3b = 0, a^2=0, b^2 = b, ab = 0 = ba \rangle. \]}
We also introduce the elements
\[
c = a + b,\quad
d = 2b,\quad
e = a + 2b.
\]
The Table~\ref{tab:add} presents the addition table common to both rings; the multiplication tables for \(H_{23}\) and \(H_{32}\) are shown in Table~\ref{tab:mult}, respectively.  
\begin{table}[h]
\centering
\begin{tabular}{lrrrrrr}
\toprule
$+$ & $0$ & $a$ & $b$ & $c$ & $d$ & $e$ \\
\midrule
$0$ & $0$ & $a$ & $b$ & $c$ & $d$ & $e$ \\
$a$ & $a$ & $0$ & $c$ & $b$ & $e$ & $d$ \\ 
$b$ & $b$ & $c$ & $d$ & $e$ & $0$ & $a$ \\
$c$ & $c$ & $b$ & $e$ & $d$ & $a$ & $0$ \\
$d$ & $d$ & $e$ & $0$ & $a$ & $b$ & $c$ \\
$e$ & $e$ & $d$ & $a$ & $0$ & $c$ & $b$ \\
\bottomrule
\end{tabular}
\caption{Addition table for $H_{23}$ and $H_{32}$.}
\label{tab:add}
\end{table}

\begin{table}[h]
\centering
\begin{minipage}{0.47\linewidth}
\centering
\label{tab:mult}
\begin{tabular}{lrrrrrr}
\toprule
$\cdot$ & $0$ & $a$ & $b$ & $c$ & $d$ & $e$ \\
\midrule
$0$ & $0$ & $0$ & $0$ & $0$ & $0$ & $0$ \\
$a$ & $0$ & $a$ & $0$ & $a$ & $0$ & $a$ \\ 
$b$ & $0$ & $0$ & $0$ & $0$ & $0$ & $0$ \\
$c$ & $0$ & $a$ & $0$ & $a$ & $0$ & $a$ \\
$d$ & $0$ & $0$ & $0$ & $0$ & $0$ & $0$ \\
$e$ & $0$ & $a$ & $0$ & $a$ & $0$ & $a$ \\
\bottomrule
\end{tabular}
\end{minipage}\hfill
\begin{minipage}{0.47\linewidth}
\centering
\begin{tabular}{lrrrrrr}
\toprule
$\cdot$ & $0$ & $a$ & $b$ & $c$ & $d$ & $e$ \\
\midrule
$0$ & $0$ & $0$ & $0$ & $0$ & $0$ & $0$ \\
$a$ & $0$ & $0$ & $0$ & $0$ & $0$ & $0$ \\ 
$b$ & $0$ & $0$ & $b$ & $b$ & $d$ & $d$ \\
$c$ & $0$ & $0$ & $b$ & $b$ & $d$ & $d$ \\
$d$ & $0$ & $0$ & $d$ & $d$ & $b$ & $b$ \\
$e$ & $0$ & $0$ & $d$ & $d$ & $b$ & $b$ \\
\bottomrule
\end{tabular}
\end{minipage}
\caption{Multiplication tables for $H_{23}$ (left) and $H_{32}$ (right).}
\label{tab:mult}
\end{table}
\newpage
Each of \(H_{23}\) and \(H_{32}\) is a commutative ring without unity and with exactly two maximal ideals $J_a = \{0,a\}, \, J_b = \{0,b,d\},$
so that
\[
H_z \;=\; J_a \oplus J_b,
\]
where $z\in\{23,32\}$. Moreover, \(J_a \cong \mathbb{F}_2\) and \(J_b \cong \mathbb{F}_3\),  with the natural \(\F_2\)-action
\[
r\cdot 0 = 0,\quad r\cdot1 = r
\quad(\forall\,r\in J_a),
\]
and similarly \(J_b\cong\mathbb{F}_3\) becomes an \(\F_3\)-module via
\[
r\cdot 0 = 0,\quad r\cdot1 = r,\quad r\cdot2 = 2r
\quad(\forall\,r\in J_b).
\]
Assume \(p\in\{2,3\}\). Then each of \(J_a,J_b\) is naturally an \(\F_p\)-module, i.e.\ for all \(r,s,t\in J_a\) (when \(p=2\)) or \(J_b\) (when \(p=3\)) one has $r\cdot(s\oplus_p t)
\;=\;
r\cdot s\;\oplus_p\;r\cdot t,$ where \(\oplus_p\) denotes addition in \(\F_p\). 
Let \(m\ge1\) be an integer and set \(n=2m\).  A\emph{ linear} code \(\C\subseteq H_z^n\) or\emph{ $H_z$-code}  is  an \(H_z\)-submodule of \(H_z^n\).  Since $H_z \;=\; J_a \;\oplus\; J_b,$ any such code decomposes as
\[
\C
\;=\;
a\,\C_a
\;\oplus\;
b\,\C_b,
\]
where \(\C_a\subseteq\F_2^n\) and \(\C_b\subseteq\F_3^n\) are respectively binary and ternary linear codes.

\begin{remark}\label{rem}
For codewords \(\mathbf{c}_1=a\,\mathbf{x}_1+b\,\mathbf{y}_1\) and \(\mathbf{c}_2=a\,\mathbf{x}_2+b\,\mathbf{y}_2\) in an \(H_z\)-code, with
\(\mathbf{x}_i\in\F_2^{2m}\) and \(\mathbf{y}_i\in\F_3^{2m}\), define the \emph {symplectic inner product} by
\begin{equation}
	\langle \mathbf{c}_1,\mathbf{c}_2\rangle_S
= a^2\,\langle \mathbf{x}_1,\mathbf{x}_2\rangle_s^{(\F_2)}
\;+\;
b^2\,\langle \mathbf{y}_1,\mathbf{y}_2\rangle_s^{(\F_3)}.
\end{equation}
 In particular,
\[
\text{over }H_{23}:\quad \langle \mathbf{c}_1,\mathbf{c}_2\rangle_S
= a\,\langle \mathbf{x}_1,\mathbf{x}_2\rangle_s^{(\F_2)}
\qquad(b^2=0,a^2=a),
\]
\[
\text{over }H_{32}:\quad \langle \mathbf{c}_1,\mathbf{c}_2\rangle_S
= b\,\langle \mathbf{y}_1,\mathbf{y}_2\rangle_s^{(\F_3)}
\qquad(a^2=0,b^2=b).
\]
\end{remark}

The \emph{symplectic dual} of \(\C\) is given by
	\[
	\C^{\perp_s} = \{ \x\in \H_z^{2m} \mid \langle \x,\x'\rangle_s = 0\ \text{for all } \x'\in \C\}.
	\]
	We say \(\C\) is \emph{symplectic self-orthogonal} \emph{(symplectic-SO)}if \(\C \subseteq \C^{\perp_s}\), and \emph{symplectic self-dual} \emph{(symplectic-SD)} if \(\C = \C^{\perp_s}\). An $\H_z$-code that is symplectic self-orthogonal and has size \(6^m\) is called \emph{symplectic quasi self-dual (symplectic-QSD)}. $\C$ is called a \emph{ symplectic nice} if $|\C||\C^{\perp_s}|=36^m$. \emph{symplectic-LCD} ( $\H_z$-code with complementary dual) if \(\C\cap \C^{\perp_s}=\{0\}\). Two  $\H_z$-codes \(\C\) and \(\C'\) are \emph{permutation equivalent} if there exists a permutation of the \(n\) coordinate positions sending \(\C\) to \(\C'\).

	\begin{example}
Let the alphabet be $H_{23}$. Consider the repetition code 
\[
R_2 = \{00,\,aa,\,bb,\,cc,\,dd,\,ee\}.
\]
It is $H_{23}$-linear and symplectic self-orthogonal. Moreover, $|R_2|=6=6^{\,m}\quad(m=1),$
so $R_2$ is \emph{symplectic-QSD}. On the other hand,
\[
R_2^{\perp_S}=\{00,\,0b,\,0d,\,b0,\,bb,\,bd,\,d0,\,db,\,dd,\,
aa,\,ac,\,ae,\,ca,\,cc,\,ce,\,ea,\,ec,\,ee\},
\]
which strictly contains $R_2$, hence $R_2$ is \emph{not} symplectic self-dual.
\end{example}

\noindent
The following example shows that not every symplectic self-orthogonal code is also Euclidean self-orthogonal; see \cite{AAABBSS} for further discussion of Euclidean self-orthogonality.

\begin{example}
Let
\[
\C = \left\{
\begin{array}{llllll}
(0,0,0,0) & (a,0,0,0) & (b,0,0,0) & (c,0,0,0) & (d,0,0,0) & (e,0,0,0)\\
(0,a,0,0) & (a,a,0,0) & (b,a,0,0) & (c,a,0,0) & (d,a,0,0) & (e,a,0,0)\\
(0,b,0,0) & (a,b,0,0) & (b,b,0,0) & (c,b,0,0) & (d,b,0,0) & (e,b,0,0)\\
(0,c,0,0) & (a,c,0,0) & (b,c,0,0) & (c,c,0,0) & (d,c,0,0) & (e,c,0,0)\\
(0,d,0,0) & (a,d,0,0) & (b,d,0,0) & (c,d,0,0) & (d,d,0,0) & (e,d,0,0)\\
(0,e,0,0) & (a,e,0,0) & (b,e,0,0) & (c,e,0,0) & (d,e,0,0) & (e,e,0,0)
\end{array}
\right\}
\subset H_{23}^4.
\]
Then $\C$ is $H_{23}$-code and  symplectic self-orthogonal and $|\C|=2^2\cdot3^2=36$, hence \emph{symplectic-QSD}. 
However, $\C$ is not Euclidean self-orthogonal; for example,
\[
\langle(a,0,0,0),(a,0,0,0)\rangle_E = a \neq 0.
\]
\end{example}

\section{Symplectic Duality }\label{sec:duality}
In this section, we establish several fundamental properties of the symplectic dual of an \(H_z\)-codes. 

\begin{theorem}\label{thmdualH23} 
Let $\C=a\, \C_a \oplus b\, \C_b$ be an $H_{23}$-code of length $n$. Then

\begin{enumerate}
\item  $\C^{\perp_S} = a \, \C_a^{\perp_S}  \oplus b \, \mathbb{F}_3^{n}$.
\item $(\C^{\perp_S})^{\perp_S}=\C$ if and only if $\C_b=\mathbb{F}_3^{n}.$

\end{enumerate}

\end{theorem}
\begin{proof}
Write an arbitrary vector \(\mathbf{y}\in H_{23}^{n}\) as $\mathbf{y} = a\,\mathbf{y}_a + b\,\mathbf{y}_b,
$ where $
\mathbf{y}_a\in\mathbb{F}_2^{n},\;\mathbf{y}_b\in\mathbb{F}_3^{n}.$
Similarly, any codeword in \(\mathcal{C}=a\,\mathcal{C}_a\oplus b\,\mathcal{C}_b\) can be written $\mathbf{c} = a\,\mathbf{u} + b\,\mathbf{v},$ where $\mathbf{u}\in\mathcal{C}_a,\;\mathbf{v}\in\mathcal{C}_b.$
From Remark \ref{rem}, we have 
\[
\langle \mathbf{y},\mathbf{c}\rangle_S
= a\,\langle \mathbf{y}_a,\mathbf{u}\rangle_s^{(\mathbb{F}_2)}.
\]
Hence
\[
\langle \mathbf{y},\mathbf{c}\rangle_S = 0\ \forall\,\mathbf{c}\in\C
\quad\Longleftrightarrow\quad
\langle \mathbf{y}_a,\mathbf{u}\rangle_s^{(\mathbb{F}_2)}=0\ \forall\,\mathbf{u}\in\C_a
\;\Longleftrightarrow\;
\mathbf{y}_a\in\C_a^{\perp_S}.
\]
There is no constraint on \(\mathbf{y}_b\), so
\[
\C^{\perp_S}
= a\,\C_a^{\perp_S}\;\oplus\;b\,\mathbb{F}_3^{n}.
\]
Applying duality again gives
\[
(\mathcal{C}^{\perp_S})^{\perp_S}
= \bigl(a\,\mathcal{C}_a^{\perp_S}\oplus b\,\mathbb{F}_3^{2m}\bigr)^{\perp_S}
= a\,(\mathcal{C}_a^{\perp_S})^{\perp_S}\;\oplus\;b\,\mathbb{F}_3^{2m}
= a\,\mathcal{C}_a\;\oplus\;b\,\mathbb{F}_3^{2m}.
\]
Comparing with \(\mathcal{C}=a\,\mathcal{C}_a\oplus b\,\mathcal{C}_b\) yields
\[
(\mathcal{C}^{\perp_S})^{\perp_S} = \mathcal{C}
\quad\Longleftrightarrow\quad
b\,\mathbb{F}_3^{2m} = b\,\mathcal{C}_b
\quad\Longleftrightarrow\quad
\mathcal{C}_b = \mathbb{F}_3^{2m}.
\]

\end{proof}
\qed
\begin{theorem}\label{thmdualH32} 
Let $\C=a\, \C_a \oplus b\, \C_b$ be an $H_{32}$-code of length $n$. Then

\begin{enumerate}
\item  $\C^{\perp_S} = a \,  \mathbb{F}_2^{n}   \oplus b \, \C_b^{\perp_S}$.
\item $(\C^{\perp_S})^{\perp_S}=\C$ if and only if $\C_a=\mathbb{F}_2^{n}.$

\end{enumerate}

\end{theorem}
\begin{proof}
By symmetry with Theorem~\ref{thmdualH23}, swapping  \(a\) and \(b\) and using Remark \ref{rem}
for \(H_{32}\), one immediately obtains $\C^{\perp_S}
= a\,\F_2^n \;\oplus\; b\,\C_b^{\perp_S},$
and dualizing again gives
\[
(\C^{\perp_S})^{\perp_S}
= a\,(\F_2^n)^{\perp_S}\;\oplus\; b\,(\C_b^{\perp_S})^{\perp_S}
= a\,\C_a\;\oplus\; b\,\C_b=\C,
\]
whence \((\C^{\perp_S})^{\perp_S}=\C\) exactly when \(\C_a=\F_2^n\).
\end{proof}

\qed


\medskip

\noindent  Combining the duality formulas of Theorems~\ref{thmdualH23} and~\ref{thmdualH32} with the definition of symplectic nice codes yields the following two corollaries.

\begin{corollary}\label{nice1} 
Let $\C=a\, \C_a \oplus b\, \C_b$ be an $H_{23}$-code of length $n=2m$. 
Then $\C$ is symplectic nice if and only if $\C_b=\{\mathbf{0}\}$.

\end{corollary}

\begin{proof}
From Theorem~\ref{thmdualH23} we have $|\C^{\perp_S}|=|\C_a^{\perp_S}|\;3^n.$
By standard symplectic‐duality over \(\F_2^n\), $\dim(\C_a) + \dim\bigl(\C_a^{\perp_S}\bigr) = n,$
hence $|\C_a|\;|\C_a^{\perp_S}|
=2^{\dim(\C_a)}\;2^{\dim(\C_a^{\perp_S})}
=2^n.$
Therefore
\[
|\C|\;|\C^{\perp_S}|
=2^n\;\cdot\;3^n\;\cdot\;|\C_b|
=6^n\;|\C_b|
=36^m\;|\C_b|.
\]
It follows that \(\C\) is symplectic nice, i.e.\ \(|\C|\;|\C^{\perp_S}|=36^m\), if and only if \(|\C_b|=1\), if and only if \(\C_b=\{\mathbf0\}\).
\end{proof}

\qed

\begin{corollary}\label{nice2} 
Let $\C=a\, \C_a \oplus b\, \C_b$ be an $H_{32}$-code of length $n=2m$. 
Then $\C$ is symplectic nice if and only if $\C_a=\{\mathbf{0}\}$.

\end{corollary}
\begin{proof}
Follows by the same argument as Corollary~\ref{nice1}, with Theorem~\ref{thmdualH32} replacing Theorem~\ref{thmdualH23}.
\end{proof}

\qed



\begin{theorem}\label{thmorthH23} 
Let $\C=a\, \C_a \oplus b\, \C_b$ be an $H_{z}$-code. then

\begin{enumerate}
\item if $z=23$, then, $\C$ is symplectic self-orthogonal if and only if $\C_a$ is a binary symplectic self-orthogonal code.
\item if $z=32$, then,$\C$ is symplectic self-orthogonal if and only if $\C_b$ is a ternary symplectic self-orthogonal code.
\end{enumerate}
\end{theorem}

\begin{proof}
Write \(\C=a\,\C_a\oplus b\,\C_b\subseteq H_z^n\).
\begin{enumerate}
\item \medskip
\noindent\textbf{Case \(z=23\).} By Theorem~\ref{thmdualH23}, we have $\C^{\perp_S}
= a\,\C_a^{\perp_S}\;\oplus\;b\,\mathbb{F}_3^{2m}.$
Hence
\[
\C\subseteq \C^{\perp_S}
\;\Longleftrightarrow\;
a\,\C_a\oplus b\,\C_b
\;\subseteq\;
a\,\C_a^{\perp_S}\oplus b\,\mathbb{F}_3^{2m}
\;\Longleftrightarrow\;
a\,\C_a\subseteq a\,\C_a^{\perp_S},
\]
  But
\[
a\,\C_a\subseteq a\,\C_a^{\perp_S}
\quad\Longleftrightarrow\quad
\C_a\subseteq \C_a^{\perp_S},
\]
so \(\C\) is symplectic self-orthogonal precisely when \(\C_a\) is symplectic self-orthogonal code over \(\mathbb{F}_2\).  
\item \textbf{Case \(z=32\).} An entirely analogous argument, using Theorem~\ref{thmdualH32} in place of Theorem~\ref{thmdualH23}, shows
\[
\C\subseteq\C^{\perp_S}
\iff
\C_b\subseteq\C_b^{\perp_S}.
\]
\end{enumerate}
\end{proof}

\qed

\par In the next two theorems, we establish a method to construct symplectic self-orthogonal codes and symplectic-QSD codes from binary and ternary linear codes. 
\begin{theorem}\label{thmbinaryternary1} Let $\C_1$ be a binary linear code of length $n=2m$ and $\C_2$ be a ternary linear code of length $n=2m$. The code $\C$ defined by
\[\C=a\,\C_1\oplus b\,\C_2\] 
is a symplectic self-orthogonal $H_{23}$-code of length $n$ where $\C_a=\C_1$ and $\C_b=\C_2$ if $\C_1$ is a binary symplectic  self-orthogonal code. Moreover, if $|\C_1||\C_2|=  6^{m}$, then $\C$ is symplectic-QSD.
\end{theorem}

\begin{proof}
Write an arbitrary pair of codewords in \(\C\) as $\mathbf{x}
= a\,\mathbf{u} + b\,\mathbf{v}$, and $\mathbf{x}' = a\,\mathbf{u}' + b\,\mathbf{v}',$
with \(\mathbf{u},\mathbf{u}'\in\C_1\) and \(\mathbf{v},\mathbf{v}'\in\C_2\).  Since \(\C_1\) is linear over \(\mathbb{F}_2\) and \(\C_2\) is linear over \(\mathbb{F}_3\), we have $\mathbf{u} + \mathbf{u}' \in \C_1,$ and $\mathbf{v} + \mathbf{v}' \in \C_2$.
Hence
\[
\mathbf{x} + \mathbf{x}'
= a\,(\mathbf{u} + \mathbf{u}') \;+\; b\,(\mathbf{v} + \mathbf{v}')
\;\in\;
a\,\C_1 \;\oplus\; b\,\C_2
= \C,
\]
showing closure under addition. Next, let \(r\in H_{23}\).  By the decomposition \(H_{23}=J_a\oplus J_b\) we write \(r=r_a+r_b\) with \(r_a\in J_a\cong\mathbb{F}_2\) and \(r_b\in J_b\cong\mathbb{F}_3\).  The action of \(r_a\) on the summand \(a\,\C_1\) coincides with the \(\mathbb{F}_2\)-scalar multiplication on \(\C_1\), and the action of \(r_b\) on \(b\,\C_2\) coincides with the \(\mathbb{F}_3\)-scalar multiplication on \(\C_2\).  Therefore
\[
r\,\mathbf{x}
= r\,(a\,\mathbf{u}) \;+\; r\,(b\,\mathbf{v})
= a\,(r_a\,\mathbf{u}) \;+\; b\,(r_b\,\mathbf{v}),
\]
with \(r_a\,\mathbf{u}\in\C_1\) and \(r_b\,\mathbf{v}\in\C_2\).  It follows that \(r\,\mathbf{x}\in\C\), so \(\C\) is closed under \(H_{23}\)-multiplication.

Together, these show \(\C\) is  a linear \(H_{23}\)-code. To see that \(\C\) is symplectic self-orthogonal, Let $\mathbf{c}=a\,\mathbf{u}+b\,\mathbf{v}$, and $\mathbf{c}'=a\,\mathbf{u}'+b\,\mathbf{v}'$
be arbitrary codewords of \(\C\).  Then using \(a^2=a\), \(ab=ba=0\), \(b^2=0\) in \(H_{23}\) one computes
\[
\bigl\langle \mathbf{c},\mathbf{c}'\bigr\rangle_S
=\bigl\langle a\,\mathbf{u}+b\,\mathbf{v},\,a\,\mathbf{u}'+b\,\mathbf{v}'\bigr\rangle_S
= a\,\bigl\langle \mathbf{u},\mathbf{u}'\bigr\rangle_s^{(\F_2)}.
\]
Since \(\C_1\) is symplectic self-orthogonal over \(\F_2\), \(\langle \mathbf{u},\mathbf{u}'\rangle_s^{(\F_2)}=0\).  Hence \(\langle \mathbf{c},\mathbf{c}'\rangle_S=0\) for all \(\mathbf{c},\mathbf{c}'\in\C\), i.e.\ \(\C\subseteq\C^{\perp_S}\).
Moreover, since \(\lvert\C\rvert=\lvert\C_1\rvert\cdot\lvert\C_2\rvert\), the condition
\(\lvert\C_1\rvert\,\lvert\C_2\rvert=6^m\) gives \(\lvert\C\rvert=6^m\).   Hence \(\C\) is symplectic-QSD.
\end{proof}
\qed
The following result for \(H_{32}\) follows by the same  argument as Theorem~\ref{thmbinaryternary} and is stated without proof.

\begin{theorem}\label{thmbinaryternary2} Let $\C_1$ be a binary linear code of length $n=2m$ and $\C_2$ be a ternary linear code of length $n=2m$. The code $\C$ defined by
\[\C=a\,\C_1\oplus b\,\C_2\] 
is a symplectic self-orthogonal $H_{32}$-code of length $n$ where $\C_a=\C_1$ and $\C_b=\C_2$ if $\C_2$ is a ternary symplectic  self-orthogonal code. Moreover, if $|\C_1||\C_2|=  6^{m}$, then $\C$ is symplectic-QSD.
\end{theorem}


\begin{lemma}\label{H23QSD}
Let $\C=a\, \C_a \oplus b\, \C_b$ be an $H_{23}$-code of length $n=2m$. Then  $\C$ is a symplectic-QSD code  if and only if  $\C_a$ is a binary symplectic self-dual code and $\C_b$ is a ternary $[n,m]$ code.	
\end{lemma}

\begin{proof}
First suppose \(\C\) is symplectic-QSD of length \(2m\) and using Theorem~\ref{thmdualH23} gives \(\C\subseteq\C^{\perp_S}\) if and only if \(\C_a\subseteq\C_a^{\perp_S}\), so \(\C_a\) is symplectic self-orthogonal over \(\mathbb F_2\).  Moreover, $|\C|=|\C_a|\;|\C_b|=6^m.$ Every symplectic self-orthogonal binary code \(\C_a\subseteq\mathbb F_2^{2m}\) satisfies \(|\C_a|\le2^m\), with equality if and only if \(\C_a=\C_a^{\perp_S}\).  Since \(|\C_a|\) divides \(6^m\), the only possibility is \(|\C_a|=2^m\).  Hence \(\dim\C_a=m\) and \(\C_a=\C_a^{\perp_S}\), i.e.\ \(\C_a\) is symplectic self-dual.  Likewise, \(|\C_b|=6^m/2^m=3^m\), so \(\C_b\) has dimension \(m\) over \(\mathbb F_3\), i.e.\ is a ternary \([2m,m]\) code.

Conversely, if \(\C_a\) is symplectic self-dual of length \(2m\) (so \(|\C_a|=2^m\)) and \(\C_b\) is a ternary \([2m,m]\) code (so \(|\C_b|=3^m\)), then $\C$ is symplectic self-orthogonal (since \(\C_a\subseteq\C_a^{\perp_S}\)) and $|\C|
=|\C_a|\;|\C_b|
=2^m\cdot3^m
=6^m.$ Thus \(\C\) is symplectic-QSD.  
\end{proof}


\begin{lemma}\label{H32QSD}
Let $\C=a\, \C_a \oplus b\, \C_b$ be an $H_{32}$-code of length $n=2m$. Then  $\C$ is a symplectic-QSD code  if and only if  $\C_a$ is  is a binary $[n,m]$ code  and $\C_b$ is a ternary symplectic self-dual code.
\end{lemma}

\begin{proof}
By symmetry with Lemma~\ref{H23QSD}, exchanging the binary and ternary parts and using Theorem~\ref{thmdualH32} in place of Theorem~\ref{thmdualH23} yields the result.
\end{proof}


\begin{theorem}\label{SD}
Let $\C=a\, \C_a \oplus b\, \C_b$ be an $H_{23}$-code of length $2m$. Then  $\C$ is symplectic self-dual if and only if  $\C_a$ is a binary symplectic self-dual  code and $\C_b=\mathbb{F}_3^n$.
\end{theorem}
\begin{proof}
By Theorem~\ref{thmdualH23} we have
\[
\C^{\perp_S}
= a\,\C_a^{\perp_S}\;\oplus\;b\,\mathbb{F}_3^{2m}.
\]
Therefore
\[
\C = \C^{\perp_S}
\;\Longleftrightarrow\;
a\,\C_a\oplus b\,\C_b
\;=\;
a\,\C_a^{\perp_S}\oplus b\,\mathbb{F}_3^{2m}
\;\Longleftrightarrow\;
\begin{cases}
a\,\C_a = a\,\C_a^{\perp_S},\\
b\,\C_b = b\,\mathbb{F}_3^{2m}.
\end{cases}
\;\Longleftrightarrow\;
\begin{cases}
\C_a = \C_a^{\perp_S},\\
\C_b = \mathbb{F}_3^{2m}.
\end{cases}
\]

\end{proof}

\begin{theorem}\label{SD2}
Let $\C=a\, \C_a \oplus b\, \C_b$ be an $H_{32}$-code of length $2m$. Then  $\C$ is symplectic self-dual if and only if  $\C_a=\mathbb{F}_3^n$ and  $\C_b$ is a ternary symplectic self-dual code.
\end{theorem}

\begin{proof}
The proof is entirely analogous to that of Theorem~\ref{SD}, but using the duality description from Theorem~\ref{thmdualH32} and interchanging the roles of the binary and ternary summands.
\end{proof}
\bigskip
Observe that by Corollary~\ref{H23QSD} and Theorem~\ref{SD},  symplectic-QSD and symplectic self-dual $H_{23}$-codes have even lengths while symplectic-QSD and symplectic self-dual $H_{32}$-codes have length divisible by 4 since $\C_b$ must be ternary symplectic self-dual codes.

Now, we will show that any $H_{23}$-code or  $H_{32}$-code cannot be both symplectic-QSD and symplectic self-dual at the same time. 
\begin{proposition}\label{QSDnotSD}
Let $\mathcal{Q}$ be the set of all QSD codes of length $n$ over $H_{23}$ (resp. $H_{32}$)and $\mathcal{S}$ be the set of all self-dual
codes of length $n$ over $H_{23}$(resp. $H_{32}$). Then $\mathcal{Q}\cap \mathcal{S}=\emptyset$.
\end{proposition}
\begin{proof}
If \(\C\in\mathcal Q\) then by Corollary~\ref{H23QSD} we have
\[
\C=a\,\C_a\oplus b\,\C_b
\quad\text{with}\quad
\dim_{\F_3}\C_b = m.
\]
If also \(\C\in\mathcal S\) then by Theorem~\ref{SD}
\[
\C_b = \F_3^{2m},
\]
so \(\dim_{\F_3}\C_b=2m\).  Since \(m>0\), \(m\neq2m\), this is impossible.  Hence no nontrivial code can lie in both \(\mathcal Q\) and \(\mathcal S\).
\end{proof}
\qed

\begin{theorem}\label{LCDH23}
Let \(\C=a\,\C_a\oplus b\,\C_b\) be an \(H_{23}\)-code of length \(n=2m\).  Then \(\C\) is symplectic LCD (i.e.\ \(\C\cap\C^{\perp_S}=\{\mathbf{0}\}\)) if and only if \(\C_a\) is binary symplectic LCD and \(\C_b=\{\mathbf{0}\}\).
\end{theorem}
\begin{proof}
By Theorem~\ref{thmdualH23} we have $\C^{\perp_S}
= a\,\C_a^{\perp_S}\;\oplus\;b\,\F_3^n.$
Hence
\begin{align*}
\C\cap\C^{\perp_S}
&=\bigl(a\,\C_a\oplus b\,\C_b\bigr)\;\cap\;\bigl(a\,\C_a^{\perp_S}\oplus b\,\F_3^n\bigr)\\
&=a\bigl(\C_a\cap\C_a^{\perp_S}\bigr)\;\oplus\;b\bigl(\C_b\cap\F_3^n\bigr)\\
&=a\bigl(\C_a\cap\C_a^{\perp_S}\bigr)\;\oplus\;b\,\C_b.
\end{align*}

\noindent
\((\Rightarrow)\) If \(\C\cap\C^{\perp_S}=\{\mathbf{0}\}\), then $a\,(\C_a\cap\C_a^{\perp_S})=\{\mathbf{0}\}
\quad\Longrightarrow\quad
\C_a\cap\C_a^{\perp_S}=\{\mathbf{0}\}$, so \(\C_a\) is symplectic LCD, and $b\,\C_b=\{\mathbf{0}\}
\quad\Longrightarrow\quad
\C_b=\{\mathbf{0}\}.$

\noindent
\((\Leftarrow)\) Conversely, if \(\C_a\cap\C_a^{\perp_S}=\{\mathbf{0}\}\) and \(\C_b=\{\mathbf{0}\}\), then $\C\cap\C^{\perp_S}
= a\{0\}\;\oplus\;b\{\mathbf{0}\}
=\{\mathbf{0}\},$ so \(\C\) is symplectic LCD.
\end{proof}
\qed

\begin{theorem}
Let $\C=a\, \C_a \oplus b\, \C_b$ be an $H_{32}$-code is symplectic LCD if and only if $\C_a=\{\mathbf{0}\}$, and $\C_b $ is ternary symplectic LCD.
\end{theorem}

\begin{proof}
The proof is identical to that of Theorem~\ref{LCDH23}, with Theorem~\ref{thmdualH32} in place of Theorem~\ref{thmdualH23}.
\end{proof}
\qed

\begin{proposition}\label{permLCDH23}
Let $\C = a\,\C_a\oplus b\,\C_b, \quad\text{and}\quad
\C' = a\,\C'_a\oplus b\,\C'_b$ be two $H_{23}$-codes of length $n=2m$.  Then:
\begin{enumerate}
  \item If $\C$ and $\C'$ are symplectic self-dual, then $\C$ and $\C'$ are permutation equivalent if and only if $\C_a$ and $\C'_a$ are permutation equivalent.
  \item If $\C$ and $\C'$ are symplectic LCD, then $\C$ and $\C'$ are permutation equivalent if and only if $\C_a$ and $\C'_a$ are permutation equivalent.
\end{enumerate}
\end{proposition}

\begin{proof}
\begin{enumerate}
\item[1.]  Suppose $\C,\C'$ are symplectic self-dual.  By Theorem~\ref{SD}, $\C_b=\F_3^n
\quad\text{and}\quad
\C'_b=\F_3^n.$ Hence $\C = a\,\C_a\;\oplus\;b\,\F_3^n,
\quad\text{and}\quad
\C' = a\,\C'_a\;\oplus\;b\,\F_3^n.$ If there is a permutation $\pi\in S_n$ with $\pi(\C)=\C'$, then for any codeword
\[
\c=a\,\u+b\,\v\in\C,\quad \u\in\C_a,\;\v\in\F_3^n,
\]
we have 
\[
\pi(\c)
=\pi(a\,\u)+\pi(b\,\v)
=a\,\pi(\u)+b\,\pi(\v)
\;\in\;
a\,\C'_a\oplus b\,\F_3^n.
\]
Projecting onto the $a$-summand shows $\pi(\u)\in\C'_a$ for all $\u\in\C_a$, hence $\pi(\C_a)=\C'_a$.

Conversely, if $\pi(\C_a)=\C'_a$, then for any 
$\c=a\,\u+b\,\v\in\C$ we get
\[
\pi(\c)
= a\,\pi(\u) + b\,\pi(\v)
\;\in\;
a\,\C'_a\oplus b\,\F_3^n
=\C',
\]
so $\pi(\C)\subseteq\C'$.  Since $\pi$ is bijective and $\dim\C=\dim\C'$, it follows $\pi(\C)=\C'$, i.e.\ $\C$ and $\C'$ are permutation equivalent.

\item[2.]  Suppose $\C,\C'$ are symplectic LCD.  By Corollary~\ref{nice1}, $\C_b=\C'_b=\{\mathbf0\}$, so 
\[
\C = a\,\C_a,
\quad
\C' = a\,\C'_a.
\]
Exactly the same permutation-projection argument shows $\pi(\C)=\C'$ if and only if $\pi(\C_a)=\C'_a$.
\end{enumerate}
\end{proof}

\begin{proposition}\label{permLCDH32}
Let $\C = a\,\C_a\oplus b\,\C_b, \quad\text{and}\quad
\C' = a\,\C'_a\oplus b\,\C'_b$ 
be two \(H_{32}\)-codes of length \(n\).  Then:
\begin{enumerate}
  \item If \(\C\) and \(\C'\) are symplectic self-dual, then \(\C\) and \(\C'\) are permutation equivalent if and only if \(\C_b\) and \(\C'_b\) are permutation equivalent.
  \item If \(\C\) and \(\C'\) are symplectic LCD, then \(\C\) and \(\C'\) are permutation equivalent if and only if \(\C_b\) and \(\C'_b\) are permutation equivalent.
\end{enumerate}
\end{proposition}

\begin{proof}
The proof is entirely analogous to that of Proposition~\ref{permLCDH23}, replacing Theorem~\ref{SD} by Theorem~\ref{SD2} and Corollary~\ref{nice1} by Corollary~\ref{nice2}.
\end{proof}
\section{Classification  }\label{Hclass}


\begin{theorem}\label{H23SDR}
Fix a compatible pair \((\C_a,\C_b)\) for \(H_z\) (with \(\C_a\) and \(\C_b\) as in the appropriate binary/ternary roles), and let $\operatorname{Aut}(\C_a),\,\operatorname{Aut}(\C_b)\;\le\;S_n$ be their permutation-automorphism groups.  Choose representatives \(\{\sigma_1,\dots,\sigma_r\}\subset S_n\) for the double cosets \(\operatorname{Aut}(\C_a)\backslash S_n/\operatorname{Aut}(\C_b)\).  Then
\[
S_{\C_a,\C_b}
\;=\;
\bigl\{\,a\,\C_a \;+\; b\,\sigma_i(\C_b)\;\bigm|\;1\le i\le r\bigr\}
\]
is a complete set of permutation-inequivalent symplectic self-orthogonal \(H_z\)-codes with component pair \((\C_a,\C_b)\), and
\[
\bigl|S_{\C_a,\C_b}\bigr|
\;=\;
\bigl|\operatorname{Aut}(\C_a)\backslash S_n/\operatorname{Aut}(\C_b)\bigr|.
\]
\end{theorem}
\noindent As an immediate consequence of Theorem~\ref{H23SDR}, analogous to Corollary 6.2 of \cite{AE}, the full classification is obtained by ranging over all inequivalent component pairs \((\C_a,\C_b)\).
\bigskip
\begin{corollary}
Let $L_a=\{\text{all inequivalent binary codes $\C_a\subseteq\F_2^n$ with symplectic self-orthogonality}\},
$\, \text{and} \,$L_b=\{\text{all inequivalent ternary codes $\C_b\subseteq\F_3^n$ with symplectic self-orthogonality}\}.$
Then the set of all symplectic self-orthogonal $H_{z}$-codes of length $n$ is
\[
\bigcup_{\C_a\in L_a,\;\C_b\in L_b}
S_{\C_a,\C_b}.
\]
\end{corollary}
\bigskip

\noindent To classify symplectic self-orthogonal \(H_z\)-codes of even length \(n=2m\le8\), we proceed with the following algorithm:

\subsubsection*{Algorithm}

\begin{enumerate}
  \item \textbf{Component choices.}  
    \begin{itemize}
      \item If \(z=23\): 
        \begin{itemize}
          \item \(L_a\): all inequivalent binary codes \(\C_a\subseteq\F_2^n\) with symplectic self-orthogonality;
          \item \(L_b=\{\F_3^n\}\) (the full ternary space).
        \end{itemize}
      \item If \(z=32\):
        \begin{itemize}
          \item \(L_a=\{\F_2^n\}\) (the full binary space);
          \item \(L_b\): all inequivalent ternary codes \(\C_b\subseteq\F_3^n\) with symplectic self-orthogonality.
        \end{itemize}
    \end{itemize}

  \item \textbf{Automorphisms.}  
    For each \((\C_a,\C_b)\in L_a\times L_b\), compute \(\operatorname{Aut}(\C_a)\) and \(\operatorname{Aut}(\C_b)\) as subgroups of \(S_n\).

  \item \textbf{Double-coset representatives.}  
    Determine a system of distinct representatives \(\{\sigma_1,\dots,\sigma_r\}\) for the double coset space
    \[
      \operatorname{Aut}(\C_a)\backslash S_n/\operatorname{Aut}(\C_b).
    \]

  \item \textbf{Code construction.}  
    For each \(i=1,\dots,r\), form
    \[
      \C^{(i)}=a\,\C_a + b\,\sigma_i(\C_b),
    \]
    and verify symplectic self-orthogonality.  These \(\C^{(i)}\) are pairwise permutation-inequivalent.

  \item \textbf{Collect.}  
    Varying \((\C_a,\C_b)\) over \(L_a\times L_b\) and aggregating all \(\C^{(i)}\) yields the complete classification for length \(n\le8\).
\end{enumerate}
\begin{example}
	
Let $
L_a=\{\langle(1,0)\rangle,\;\langle(1,1)\rangle\}\subseteq\F_2^2,$ and $L_b=\{\langle(1,0)\rangle,\;\langle(1,1)\rangle,\;\langle(1,2)\rangle\}\subseteq\F_3^2.$
By Theorem~\ref{thmbinaryternary1} every \(H_{23}\)-code of the form \(\C=a\,\C_a\oplus b\,\C_b\) with \(\C_a\in L_a\) and \(\C_b\in L_b\) is symplectic self-orthogonal. Set \(A_1=\langle(1,0)\rangle\), \(A_2=\langle(1,1)\rangle\) in \(\F_2^2\); both are symplectic self-orthogonal.  Let \(B_1=\langle(1,0)\rangle\), \(B_2=\langle(1,1)\rangle\), \(B_3=\langle(1,2)\rangle\) in \(\F_3^2\). As subgroups of \(S_2=\{e,(1\ 2)\}\) one has:
\[
\operatorname{Aut}(A_1)=\{e\},\quad \operatorname{Aut}(A_2)=S_2;
\qquad
\operatorname{Aut}(B_1)=\{e\},\quad \operatorname{Aut}(B_2)=S_2,\quad \operatorname{Aut}(B_3)=S_2.
\]

Now, compute \(\operatorname{Aut}(A_i)\backslash S_2/\operatorname{Aut}(B_j)\):
\begin{itemize}
  \item \((A_1,B_1)\): \(1\backslash S_2/1\) has size 2, giving two inequivalent choices (\(\sigma=e\) and the swap) on \(B_1\).
  \item \((A_1,B_2)\), \((A_1,B_3)\): \(1\backslash S_2/S_2\) gives 1 each.
  \item \((A_2,B_1)\), \((A_2,B_2)\), \((A_2,B_3)\): \(S_2\backslash S_2/*\) gives 1 each.
\end{itemize}
In total there are \(2+1+1+1+1+1=7\) double-coset representatives, hence 7 permutation-inequivalent codes.
The seven inequivalent symplectic self-orthogonal \(H_{23}\)-codes are
\[
\begin{array}{ll}
\C_1 = a\,\langle(1,0)\rangle \oplus b\,\langle(1,0)\rangle, &
\C_2 = a\,\langle(1,0)\rangle \oplus b\,\langle(0,1)\rangle,\\
\C_3 = a\,\langle(1,0)\rangle \oplus b\,\langle(1,1)\rangle, &
\C_4 = a\,\langle(1,0)\rangle \oplus b\,\langle(1,2)\rangle,\\
\C_5 = a\,\langle(1,1)\rangle \oplus b\,\langle(1,0)\rangle, &
\C_6 = a\,\langle(1,1)\rangle \oplus b\,\langle(1,1)\rangle,\\
\C_7 = a\,\langle(1,1)\rangle \oplus b\,\langle(1,2)\rangle. &
\end{array}
\]

\end{example}
\begin{lemma}[Lemma 3.1 \cite{}]
Let \(\langle\cdot,\cdot\rangle_s\) be a symplectic inner product on \(\mathbb{F}_p^{2m}\). Then for every vector \(v\in\mathbb{F}_p^{2m}\) one has
\[
\langle v,v\rangle_s=0.
\]
In other words, every vector is self-orthogonal with respect to the symplectic form.
\end{lemma}

\begin{lemma}[Lemma 3.2]
The number of \(k\)-dimensional symplectic self-orthogonal subspaces (i.e., totally isotropic \(k\)-spaces) of \(\mathbb{F}_p^{2m}\) is
\[
\frac{\prod_{i=0}^{k-1}\bigl(p^{2m-2i}-1\bigr)}
{\prod_{j=1}^k (p^j-1)}
\;=\;
\frac{(p^{2m-2k+2}-1)(p^{2m-2k+4}-1)\cdots(p^{2m}-1)}
{(p^k-1)(p^{k-1}-1)\cdots(p-1)}.
\]
\end{lemma}
\section{Numerical Results}\label{Numerical}
\subsection{symplectic self-orthogonal $H_{23}$-codes}
I will add them later on.

\end{document}